\documentclass[a4paper,11pt]{amsart}

\usepackage{amsmath}
\usepackage{amssymb,amsbsy,amsmath,amsfonts,amscd}
\usepackage{latexsym}

\usepackage{enumitem}
\usepackage{color,url}

\usepackage{hyperref} %

 \usepackage[left=2cm,right=2cm, top=3cm, bottom=3cm]{geometry}

\newcommand\Oh{{\mathcal O}}
\newcommand\sC{{\mathcal C}}

\newcommand\sE{{\mathcal E}}

\newcommand\sL{{\mathcal L}}

\newcommand\sB{{\mathcal B}}

\newcommand\sK{{\mathcal K}}

\newcommand\sH{{\mathcal H}}

\newcommand\al{\alpha}
\newcommand\be{\beta}
\newcommand\e{\epsilon}
\newcommand\s{\sigma}

\newcommand\De{\Delta}
\newcommand\ga{\gamma}

\DeclareMathOperator{\Alb}{Alb}

\DeclareMathOperator{\Tors}{Tors}

\DeclareMathOperator{\Ad}{Ad}

\newcommand{\CC}{\ensuremath{\mathbb{C}}}

\newcommand{\ZZ}{\ensuremath{\mathbb{Z}}}
\newcommand{\QQ}{\ensuremath{\mathbb{Q}}}

\newcommand{\hol}{\ensuremath{\mathcal{O}}}

\newcommand{\PP}{\ensuremath{\mathbb{P}}}

\newcommand{\ra}{\ensuremath{\rightarrow}}

\def\eea{\end{eqnarray*}}
\def\bea{\begin{eqnarray*}}

\DeclareMathOperator{\Id}{Id}
\DeclareMathOperator{\Aut}{Aut}

\DeclareMathOperator{\Inn}{Inn}

\DeclareMathOperator{\Num}{Num}

\newcommand\dual{\mathrel{\raise3pt\hbox{$\underline{\mathrm{\thinspace d
\thinspace}}$}}}
\newcommand\qe{\ifhmode\unskip\nobreak\fi\quad $\Box$}       

\def\BOX{\hfill\lower.5\baselineskip\hbox{$\Box$}}

\newtheorem{theorem}{Theorem}

\newtheorem{question}[theorem]{Question}

\newtheorem{ex}[theorem]{Example}

\newtheorem{remark}[theorem]{Remark}

\theoremstyle{definition}
\newtheorem{defin}[theorem]{Definition}
\newenvironment{definition}{\begin{defin}\rm}{\end{defin}}

\numberwithin{equation}{section}

\usepackage{listings} 
 \lstdefinelanguage{Magma}%
  {%
   otherkeywords={:=,+:=,-:=,*:=},%
   procnamekeys={function,func,intrinsic,procedure,proc,return},%
   morekeywords={true,false},%
   morekeywords=[2]{adj,and,cat,cmpeq,cmpne,diff,div,eq,ge,gt,in,is,join,le,lt,%
          meet,mod,ne,notadj,notin,notsubset,or,sdiff,subset,xor},%
   morekeywords=[3]{assigned,break,by,case,catch,continue,declare,default,%
          delete,do,elif,else,end,eval,exists,exit,for,forall,fprintf,if,local,%
          not,print,printf,quit,random,read,readi,repeat,restore,save,select,%
          then,time,to,try,until,vprint,vprintf,vtime,when,where,while},%
   morekeywords=[4]{clear,forward,freeze,iload,import,load},%
   morekeywords=[5]{assert,assert2,assert3,error,require,requirege,requirerange},%
   morekeywords=[6]{car,comp,cop,elt,ext,frac,hom,ideal,iso,lideal,loc,map,%
          ncl,pmap,quo,rec,recformat,rep,rideal,sub},%
      sensitive,%
      morecomment=[l]//,%
      morecomment=[s]{/*}{*/},%
      morestring=[b]"%
  }[keywords,procnames,comments,strings]%

\lstnewenvironment{code_magma}[1][]
{\lstset{basicstyle=\scriptsize\ttfamily, columns=fullflexible,
language=Magma,
keywordstyle=\color{red}\bfseries,
commentstyle=\color{blue},tabsize=6,
numbers=left, numberstyle=\tiny,
stepnumber=1, numbersep=5pt, frame=single, #1}}{}

\begin{document}

\title[Coh. or num.  trivial automorphisms of surfaces of general type]{Cohomologically or numerically  trivial automorphisms of surfaces of general type}
\author[F.~Catanese, D.~Frapporti]{Fabrizio  Catanese, Davide Frapporti}
\address{Lehrstuhl Mathematik VIII, Mathematisches Institut der Universit\"{a}t
Bayreuth, NW II, Universit\"{a}tsstr. 30,
95447 Bayreuth.}
\email{Fabrizio.Catanese@uni-bayreuth.de}
\address{Politecnico di Milano, Dipartimento di Matematica, via Bonardi 9,
20133 Milano}
\email{davide.frapporti@polimi.it}


\thanks{
The second author is a member of G.N.S.A.G.A. of I.N.d.A.M.
}

\maketitle

\begin{abstract}
Our main result is the determination of the respective groups 
$  \Aut_\ZZ(S) $ of cohomologically  trivial automorphisms
and $  \Aut_\QQ(S) $ of numerically  trivial automorphisms
for the reducible fake quadrics, that is, the surfaces $S$ isogenous to a product with $q=p_g=0$.

In this way we produce new record winning examples: a surface $S$ with $|\Aut_\QQ(S)| =192$,
and a surface whose cohomology has torsion with  nontrivial $  \Aut_\ZZ(S) \cong \ZZ/2.$
\end{abstract}

\tableofcontents

\section*{Introduction}

As in \cite{CatLiu21}  let $  \Aut_R(S): = \{ g \in \Aut(S)| g^* = 1 \ on\ H^*(S,R)\}$ denote the group of those automorphism of  $S$ acting trivially on the cohomology with coefficients in the ring $R$.

In this paper  we continue (cf.~\cite{CatLiu21}, \cite{CFGLS24}, \cite{CLS24})  the study of \textit{cohomologically trivial automorphisms}    $\Aut_\ZZ(S)$  and
\textit{numerically trivial automorphisms}   $\Aut_\QQ(S)$  of surfaces.  The former group $\Aut_\ZZ(S)$ is a  subgroup of the latter, 
and for minimal non-ruled surfaces there are no examples where the group $\Aut_\ZZ(S)$
is finite with cardinality $\geq 4$;  the
first authors who dealt with the subtleties of the  subgroup $\Aut_\ZZ(S)$  were Mukai and Namikawa \cite{MN84}.

The basic questions  we consider in this article  are the followings.

\begin{question}
 (I) What is  $ max | \Aut_\ZZ(S)| $  for a minimal surface $S$  of general type?
 
 (II)  What is  $ max | \Aut_\QQ(S)| $  for a minimal surface $S$  of general type?
\end{question}

Since $\Aut_\ZZ(S) \subset \Aut_\QQ(S)$, 
that such  maxima exist follows from a result of Xin Jing Cai \cite{Cai04} (beware that  Cai uses the term cohomologically trivial for  the weaker notion, introduced later, of  numerically trivial).

\begin{theorem}[Cai]
For  surfaces $S$ of general type, there is an absolute constant $C$ such that
$$ | \Aut_\QQ(S)| \leq C .$$
\end{theorem}



There is a well known  example for question (II): for the classical Beauville surface $ | \Aut_\QQ(S)| = 75 $ and $ | \Aut_\ZZ(S)| = 1$,
as we shall show later on\footnote{And as shown by Gregorz Gromadzki in the preprint \cite{cg}.}.

Recall that the above  is a surface isogenous to a product, $S = (C\times C) / G$, $ G \cong (\ZZ/5)^2$, where $C$ is the Fermat quintic
curve, and the action is free. $\Aut_\QQ(S)$  is induced by $(G  \rtimes \ZZ/3) \times \Id$, and we prove that 
in this case $ | \Aut_\ZZ(S)| = 1 .$

 We  produce in this paper an example of a  surface isogenous to a product  with $| \Aut_\QQ(S)|=192$,
 hence we ask the following.
\begin{question}
 Is   $ \max | \Aut_\QQ(S)| = 192 $  for a minimal surface $S$  of general type?
\end{question}

The second  aim  that we accomplish here  is  to provide   an example of a surface of general type  $S$ with $ | \Aut_\ZZ(S)| = 2 ,$ 
and $H^*(S,\ZZ)$ having  nontrivial torsion (an example where there was no torsion had been previously obtained by Cai \cite{cai}).

More precisely, this is our main result, summarizing  Theorems \ref{record} and \ref{AutZ0}.

\begin{theorem}\label{ex-record}
Let $S=(C_1\times C_2)/G$ be a surface isogenous to a product of unmixed type, with $q(S)= p_g(S)=0$.

Then  $|\Aut_\QQ(S)| \leq  192$,
and equality is attained for $G = (\ZZ/2)^3$. 

If $G$ is abelian,  $|\Aut_\QQ(S)| $ reaches    $160$ for $G = (\ZZ/2)^4$,  
$72$ for $G = (\ZZ/3)^2$,  and  $|\Aut_\QQ(S)| =  75$ for the classical Beauville surface; while
$|\Aut_\QQ(S)| \leq  32$ for   $G$ not abelian.

Moreover,   $\Aut_\ZZ(S)$ is trivial, except for the case of $G=D_4\times \ZZ/2$, where  $\Aut_\ZZ(S) \cong \ZZ/2$
and the intersection form is even.

\end{theorem}

The reader may ask for the reasons to consider surfaces isogenous to a product.

There are several, all pointing out to the phenomenon that, when    $|\Aut_\QQ(S)| $ achieves its maximum,
then $S$ is isogenous to a product.

We review in fact in Section \ref{irregular} the status of our two main questions for irregular surfaces, that is, with $q>0$, i.e., $H_1(S, \ZZ)$ infinite.

In  \cite{clz} and \cite{Liu18} it is proved that irregular surfaces with $q(S) \geq 2$ and nontrivial $\Aut_\QQ(S)$ are necessarily
isogenous to a product of unmixed type, with $q(S)=2$; and then $\Aut_\QQ(S) = \ZZ/2$.

 Later \cite{CL18} proved that irregular surfaces with $q(S)=1$ have   $|\Aut_\QQ(S)| \leq 4$ 
 and if equality holds, $S$ is  necessarily
isogenous to a product of unmixed type.

 Finally, in \cite{Cai04} Cai showed that  $|\Aut_\QQ(S)| \leq 4$  if $\chi(\mathcal O_S)> 188$, and recently in \cite{CL24} the case $|\Aut_\QQ(S)| = 4$   has been investigated for regular surfaces. In particular, they showed that  if  $K_S$ is ample and equality is attained then 
$\Aut_\QQ(S) = (\ZZ/2)^2$ and  $S$ is isogenous to a product of curves of unmixed type. Moreover, they  provide an infinite series of examples having $K^2$ arbitrary high.


These are good reasons for us to  investigate in this article  the groups $ \Aut_\ZZ(S),  \Aut_\QQ(S)$ 
for surfaces isogenous to a product.

  \section{Surfaces isogenous to a higher product}
  
  We recall some  definitions and facts essentially introduced in  \cite{isogenous},  and 
we refer to it  for further details.
  
%
%
  
  \begin{defin}
  A smooth surface is said to be {\bf isogenous to a (higher) product} if there is a finite group $G$ and 
  curves $C_1, C_2$ of respective genera $g_1, g_2 \geq 2$ such that $G$ acts freely on $C_1 \times C_2$,
  and such that 
   $ S = (C_1 \times C_2)/G$.
  
  \end{defin}

The word ``higher'' emphasises that the respective genera of $C_1, C_2$ are $\geq 2$, in particular this implies that $S$ is of general type (and with ample canonical divisor).
However, for commodity, from now on we shall drop  the  word ``higher''.

There are two types of surfaces  isogenous to a  product:
the \textit{mixed type} is when there are elements of $G$ swapping the factors; and 
the \textit{unmixed type} is when $G$ acts on each curve and  diagonally on the product.

  
  \begin{remark}\label{rem6}
  (I) If $S$ is isogenous to a product, it is of unmixed type unless possibly if $C_1 \cong C_2 \cong C$.
  
  (II) In the mixed  case there is an exact sequence $ 1 \ra G^0 \ra G \ra \ZZ/2$ such that 
  $$S^0 : = (C \times C)/G^0$$
  is of unmixed type, and there is an element in $ G \setminus G^0$ which is an automorphism of the form 
  $\tau' (x,y) = (y , \tau(x))$, with $\tau \in G^0$.
  
  (III) In all cases we can take a minimal  realization  $ S = (C_1 \times C_2)/G$, this means that 
  no element of $G$ acts trivially on one of the two curves.
  
  (IV) We have an exact sequence of fundamental groups associated to the minimal realization:
  $$ 1 \ra \pi_{g_1} \times \pi_{g_2}  \ra \pi_1(S) \ra G \ra 1,$$
  which is left invariant by every automorphism of $S$.
  
 By the lifting condition (a continuous map lifts to an unramified  covering space  if and only if it leaves invariant the associated subgroup) every automorphism of $S$ lifts to the minimal realization $C_1 \times C_2$.
  
  \end{remark}

  The conclusion of Remark \ref{rem6}-(IV)  is that $\Aut(S)$ lifts to $C_1 \times C_2$, and we see that 
  $$ \Aut(S) = N(G) / G , \qquad N(G) < \Aut (C_1 \times C_2) \text{  being  the normalizer  of } G.$$

From now on we restrict to surfaces isogenous to a product of  \textit{unmixed type}.

  \begin{defin}
  (i) Denote for simplicity $A_1 : = \Aut(C_1), A_1 : = \Aut(C_2)$.
  
  (ii) say moreover that we are in the \textit{strongly unmixed case} if $C_1$ is not isomorphic to $C_2$.
  
  \end{defin}
  
  Then, in the strongly unmixed case
  $\Aut (C_1 \times C_2) = A_1 \times A_2$, whereas if $\s(x,y) = (y,x)$, then 
  $$\Aut (C \times C)  = A^2 \rtimes \langle \s \rangle  \cong  A^2 \rtimes \ZZ/2. $$

 \subsection{Group-theoretical description}
 
The description of surfaces isogenous to a product  is accomplished through the Riemann existence theorem,
which allows to construct 
  Galois coverings between projective curves (we refer to  \cite{Survey} for more details on this part).

\begin{definition}\label{gv}
Let $H$ be a finite group, and let  $g\geq 0 \mbox{ and  } m_1, \ldots, m_r > 1$
be integers.
A \textit{Hurwitz generating vector} for $H$ of type $(g;m_1,\ldots ,m_r)$ is a $(2g+r)$-tuple of
elements of $H$:
$$V:=(a_1,b_1,\ldots, a_{g},b_{g};c_1, \ldots, c_r)$$
such that 
\begin{itemize}
\item the elements of $V$ generate $H$, 
\item $\prod_{i=1}^{g}[a_i,b_i]\cdot c_1\cdot c_2\cdots c_r=1$ and 
\item  $\mathrm{ord}(c_i)=m_i$ for $i=1,\ldots, r$.
\end{itemize}
\end{definition}

By Riemann's existence theorem, any curve $C$  together with an action of
a finite group $H$ on it, such that $C/H$ is a curve $C'$ of genus $g$, is essentially determined by 
a generating vector $V$ for $H$ of signature $(g;m_1,\ldots ,m_r)$, where the mapping class group of the punctured base curve acts,
in particular the orders $m_1, \dots, m_r$ are defined only up to 
a permutation $\sigma \in \mathfrak{S}_r$.

We call the elements $a_1,b_1,\ldots, a_{g},b_{g}$ the \textit{global monodromies} (they correspond to the choice of a basis of $\pi_1(C')$) while $c_1,
\ldots, c_r$ are called the  \textit{local monodromies} and correspond to geometric loops around the branch points of $C\to C/H$. Moreover, they determine  the set of elements of $H$ which fix points on $C$:
$$\Sigma(V):= 
\bigcup_{h\in G}
\bigcup_{i=1}^r
\bigcup_{j=1}^{\mathrm{ord}(c_i)}
 \{ h\cdot c_i^j\cdot h^{-1}\}\,,$$
is called the set of stabilizers for the action of $H$ on $C$.

A surface $S=(C_1\times C_2)/G$  isogenous to a product of unmixed type is then described via a  pair $(V_1,\, V_2)$ of
generating vectors of $G$, which are \textit{disjoint}, in the sense that  $\Sigma(V_1)\cap \Sigma(V_2)=\{ 1\}$.

  \subsection{The unmixed case} To avoid confusion, since $G$ acts diagonally on  $C_1 \times C_2$,
  we view $ G < A_j$, for $j=1,2$,  and $G $ acting on $C_1 \times C_2$ as the diagonal subgroup $ \De_G < A_1 \times A_2$.
  
  Then look at the normalizer $N^0(G): = N(G) \cap (A_1 \times A_2)$ of $ \De_G  < G \times G < A_1 \times A_2$ inside $  A_1 \times A_2$.
  
  We have, setting $ \Ad(\ga) g : = \ga g \ga^{-1}$,  and letting $N_ {A_j}(G) $ be the normalizer of $G$ inside $A_j$,
  \begin{equation}
  N^0(G) = \{ (\ga_1, \ga_2)  \in  N_ {A_1}(G)\times  N_ {A_2}(G) \mid \Ad(\ga_1) = \Ad(\ga_2)\} 
   < N_ {A_1}(G) \times N_ {A_2}(G).
  \end{equation}
  Clearly $\De_G < N^0(G)$ and we have an exact sequence of groups 
  $$ 1 \ra \De_G \ra  N^0(G) \cap (G \times G) \ra Z(G) \ra 1,$$
where $Z(G)$ is the centre of $G$ and the last map $ G \times G \ra G, \ (g_1, g_2) \mapsto g_1^{-1} g_2 $ is a map of sets 
(if $G$ is not abelian),
but a homomorphism if $g_1^{-1} g_2  \in Z(G)$.

In other words, setting $N_j : = N_ {A_j}(G)$, $N^0(G)$ is the inverse image of the diagonal $\De_{\Aut(G)}$ under the exact sequence
$$ 1 \ra \sC_1 \times \sC_2 \ra N_1 \times N_2 \ra H_1 \times H_2 < \Aut(G)^2 ,$$ 
where $ \sC_j $ is the centralizer of $G$ in $A_j$, $H_j : = N_j /\sC_j$. 

Summing up, if we set $H := H_1 \cap H_2$, 
we have the exact sequence 
$$ 1 \ra \sC_1 \times \sC_2 \ra N^0(G)  \ra H \ra 1.$$ 
Dividing by $\De_G$ and defining $\Aut^*(S) : = N^0(G) / \De_G$  we get
\begin{equation}\label{eq_autS}
    1 \ra  (\sC_1 \times \sC_2) / \De_{Z(G)} \ra \Aut^*(S) \ra H / \De_{\Inn(G)} \ra 1,
\end{equation} 
since $(\sC_1 \times \sC_2) \cap \De_G = \De_{Z(G)} \cong Z(G)$ and $G / Z(G) = \Inn(G)$.

 \bigskip
 Observe that in the strongly unmixed case, $\Aut^*(S) = \Aut(S).$ 
 
 \subsection{The non strongly unmixed case}
 In the unmixed case with $C_1 \cong C_2$, we have that $N^0(G) < N(G)$ has index at most 2,
 and equal to $2$ if and only if there is an automorphism $$ \ga (x,y) = (\ga_1(y), \ga_2(x))$$
 normalizing the subgroup $\De_G$, which acts on $C_1 \times C_2$ via $g(x,y) = (g(x), g(y))$
 (but where we do not identify $C_1$ with $C_2$ under a given  isomorphism, hence we do not get the diagonal $\De_G$
 of $G \times G$). 
 
 This means that $\ga_1 : C_2 \ra C_1, \ga_2 : C_1 \ra C_2,$ satisfy that for each $g\in G$
 there is $g'$ such that $\ga_2^{-1} g \ga_2(x) = g'(x)$, $\ga_1^{-1} g \ga_1(y) = g'(y)$.
 
 To make things more concrete, identify now $C_1$ and $C_2$ with the same curve $C$ in view of a chosen isomorphism.
 
 Then we get an automorphism $\psi : A : = \Aut(C) \ra A$ such that $$\De_G = \{ (g , \psi (g))| g \in G\}.$$
 Now  $\ga (x,y) = (\ga_1 (y), \ga_2 (x))$ normalizes $\Delta_G$ if and only if  
 $$\ga_1^{-1} g  \ga_1 = \psi (\ga_2^{-1} \psi(g ) \ga_2).
 $$

 We do not carry out further these calculations, since in the unmixed case if we consider numerically trivial
 automorphisms, which  cannot exchange the two fibrations of $S$, hence they lie in $\Aut^*(S)$.
 
 \subsection{Determining the normalizers}

Recall that $ N_j < A_j : = \Aut(C_j)$ is the normalizer of $G$, hence $G < N_j $ and
 $ N'_j : = N_j /G$  descends to the  subgroup 
of automorphisms of $C'_j = C_j /G$ which lift to $C_j$.

Let $\sB$ be the branch locus of the quotient map $C_j \ra C'_j$, so that the covering is given by
an epimorphism $\pi_1 (C'_j \setminus \sB , y_0) \ra G$.

Then 
\begin{equation}
N'_j : = \{ \phi \in \Aut(C'_j) \mid \phi (\sB) = \sB, \ \exists \Phi \in \Aut(G), \Phi  \circ \mu_0 = \mu_1 \circ \phi_* \},
\end{equation}
Here
$$ \phi_* : \pi_1 (C'_j \setminus \sB , y_0) \ra \pi_1 (C'_j \setminus \sB , \phi (y_0)),$$
and $\mu_0, \mu_1$ are the respective monodromies with base points $y_0$, respectively $y_1 : = \phi (y_0)$.

Since changing the path the isomorphism of the two fundamental groups changes up to inner automorphism, 
$\Phi$ is only well defined up to $ \Inn(G)$.

The set of local monodromies of the covering is the so called Nielsen assignment of the  conjugacy class
 $ Conj (\mu(b))$  (in $G$)
to a point $b \in \sB$. And $\phi$ induces a permutation of $\sB$ preserving the Nielsen assignment.

$\mu_0$ determines the orbifold fundamental group sequence
$$ 1 \ra \pi_1(C_j , x_0) \ra \pi_1^{orb} (C'_j, y_0) \ra G \ra 1,$$
which in turn determines, via $\phi_*, \Phi, $ the action of $N'_j$ on $H_1(C_j, \ZZ)$.

\subsection{Determining the centralizers}
The elements of the centralizer of $G$ inside $A_j$ induce an automorphism of $C_j$ commuting with the action of $G$,
hence descending to
$C'_j$, and leaving fixed  the monodromy. 

In particular, if the local monodromies are 
pairwise distinct, then these elements induce the identity on the branch set.
Hence, if the base curve $C'_j$ has genus $\leq 1$, then these elements induce the identity on $C'_j$,
hence $\sC_j = Z(G)$.

\subsection{The case of cohomologically trivial automorphisms}
To simplify things, assume that we are concentrating on $\Aut_{\ZZ}(S)$. 
Assume that we have an automorphism $\ga \in \Aut_{\ZZ}(S)$. Then it induces an automorphism $\phi$ of
$C'_j$, which is in $N'_j$. $\phi$ is then acting trivially on the cohomology of $C'_j$,
and we are going to show that $\phi$ must be the identity,  with three  possible exceptions, where $N''_j$,
the image of  $\Aut_\ZZ(S)$ in $N'_j$, has order equal to 2.

By Lefschetz' theorem, this is clear if $C'_j$ has genus at least $2$.

 If instead the genus of $C'_j$ is 1, then $\phi$  could be a translation,
but then it should permute the branch points in $\sB$. By   \cite[Principle 3, page 185]{CatLiu21}, there can be only two branch points,
with multiplicity $m=2$, and the translation should be of order $2$.

\begin{ex}[{\bf Exception I}] The group $G$ is generated by elements $a,b,c : = c_1$, which are the monodromy images
$a: = \mu(\al), b: =  \mu(\be), c_1 := \mu (\ga_1)$ $$ \mu : \langle \al, \be, \ga_1, 
\ga_2| \ga_1^2 = \ga_2^2= [\al, \be]\ga_1 \ga_2 =1\rangle \ra G.$$
And there exists an automorphism $\Phi : G \ra G $ such that 
$$\Phi (a)=a, \Phi (b)=b, \Phi (c_1)=c_2, \Phi (c_2) = c_2^{-1} c_1 c_2.$$ 
\end{ex}

If instead the genus of $C'_j$ is 0, then $\phi$ can exchange, again by the cited principle,  at most 
two points of $\sB$, which have multiplicities $2$, while the other have odd multiplicity and are fixed.
Hence we have at most 4 branch points.

\begin{ex}[{\bf Exception II}] The group $G$ is generated by elements $c_1, c_2,c_3$, which are the monodromy images
$c_1: = \mu(\ga_1), c_2 := \mu (\ga_2), c_3 := \mu (\ga_3)$ $$ \mu : \langle \ga_1, 
\ga_2, \ga_3 | \ga_1^2 = \ga_2^2= \ga_1 \ga_2 \ga_3 =\ga_3^{2n+1} = 1\rangle \ra G.$$
And there exists an automorphism $\Phi : G \ra G $ such that 
$$\Phi (c_3)=c_3 , \Phi (c_1)=c_2, \Phi (c_2) = c_2^{-1} c_1 c_2 = c_3 c_2.$$ 
Observe that the Abelianization of $G$ is $\ZZ/2$, with cyclic kernel of order exactly $2n+1$.

This case exists, since the group $ \ZZ/2 * \ZZ/2$ is generated by elements $c'_1, c'_2$ of order two,
and since $c_3= c_2 c_1 $, it suffices to take  $ (\ZZ/2 * \ZZ/2 )/ \sK$, with $\sK$ normally generated by
$ (c'_2 c'_1)^{2n+1}$.
\end{ex}

\begin{ex}[{\bf Exception III}] The group $G$ is generated by elements $c_1,  c_3, c_4 $, which are the monodromy images
$c_1: = \mu(\ga_1), c_3 := \mu (\ga_3), c_4 := \mu (\ga_4)$ $$ \mu : \langle \ga_1, 
\ga_2, \ga_3, \ga_4  | \ga_1^2 = \ga_2^2= \ga_1 \ga_2 \ga_3 \ga_4 =\ga_3^{2n+1} = \ga_4^{2h+1} =1\rangle \ra G.$$
And there exists an automorphism $\Phi : G \ra G $ such that 
$$\Phi (c_3)=c_3 , \Phi (c_4)=c_4 , , \Phi (c_1)=c_2, \Phi (c_2) = c_2^{-1} c_1 c_2 .$$ 

Observe that in the Abelianization $G^{ab}$  of $G$ (multiply by $(2n+1)(2h+1)$) 
$[c_1]= [c_2]$, hence  $[c_3]= -[c_4]$, therefore $G^{ab}$  is a quotient of $\ZZ/2 \oplus \ZZ / (2r+1)$, 
where $2r + 1 = \mathrm{GCD} (2n+1, 2h+1)$ and $G^{ab}$ can be equal to $\ZZ/2 \oplus \ZZ / (2r+1)$.
\end{ex}

\begin{remark}\label{trivial}
We have $q(S) = g'_1 + g'_2$, where $g'_j $ is the genus of $C'_j = C_j /G$.

In the case where $q(S)=0$, then we can apply  \cite[Lemma 2.1]{CFGLS24} and infer that all the branch points are left fixed,
hence  $N''_1, N''_2$ are trivial.

More generally, if for instance $g'_2=0$, then $N''_1$ is trivial since the relative irregularity of the fibration $S\to C'_1$ is zero (and similarly, if $g'_1=0$, then  $N''_2$ is trivial).
\end{remark}

\section{Surfaces isogenous to a product  with $q= p_g=0$}

Surfaces isogenous to a product   with $q= p_g=0$  carry two  fibrations overs curves (see \cite{fl} for the mixed case).  Therefore, in the unmixed case, $\Aut_\QQ(S)$ equals the group
$\Aut^*(S)$ which does not interchange these two fibrations, namely  the natural isotrivial fibrations
$S\to C'_j:=C_j/G$, $j=1,2$.
In fact, $H^j(S, \QQ)=0$ for $j=1,3$, while $H^2(S, \QQ)$ has rank 2 and is generated by the classes 
of the respective fibres of the two fibrations. 

We want to study $\Aut_\QQ(S)$ and $\Aut_\ZZ(S)$
for surfaces isogenous to a product of unmixed type   with $q= p_g=0$.
These surfaces 
have been classified in \cite{BCG08}:

\begin{theorem}
Let $S = (C_1\times C_2)/G$ be a surface isogenous to a product of unmixed type,
with $p_g(S) = 0$, then $G$ is one of the groups in the Table \ref{tab1} and the types are listed in the table.
For each case in the list we have an
irreducible component of the moduli space, whose dimension is denoted by $D$.
\begin{table}[!h]
\begin{tabular}{c|c|c|c|c|c}
$G$ & $Id(G)$ & $T_1$ & $T_2$ & $D$&  $H_1(S,\mathbb Z)$\\
\hline\hline
$\mathfrak{A}_5$ & $\langle 60,5\rangle$& $[2,5,5]$& $[3,3,3,3]$&1&$(\mathbb Z/3)^2\times (\mathbb Z/{15})$\\
$\mathfrak{A}_5$ & $\langle 60,5\rangle$& $[5,5,5]$& $[2,2,2,3]$  &1&$(\mathbb Z/{10})^2$\\
$\mathfrak{A}_5$ & $\langle 60,5\rangle$& $[3,3,5]$& $[2,2,2,2,2]$  &2&$(\mathbb Z/2)^3\times \mathbb Z/6$\\
$\frak S_ 4 \times \mathbb Z/2$& $\langle 48,48 \rangle$& $[2,4,6]$& $[2,2,2,2,2,2]$  &3&
$(\mathbb Z/2)^4\times \mathbb Z/4$\\
G(32) & $\langle 32,27 \rangle$& $[2,2,4,4]$& $[2,2,2,4]$ &2&$(\mathbb Z/2)^2\times \mathbb Z/4\times \mathbb Z/8$\\
$(\mathbb Z/5)^2$ & $\langle 25,2\rangle$& $[5,5,5]$& $[5,5,5]$ &0&$(\mathbb Z/5)^3$\\
$\mathfrak {S}_4$ & $\langle 24,12\rangle$& $[3,4,4]$& $[2,2,2,2,2,2]$ &3&$(\mathbb Z/2)^4\times \mathbb Z/8$\\
G(16) & $\langle 16,3\rangle$& $[2,2,4,4]$& $[2,2,4,4]$&2&$(\mathbb Z/2)^2\times \mathbb Z/4\times \mathbb Z/8$\\
$D_4\times \mathbb Z/2$ & $\langle 16,11\rangle$& $[2,2,2,4]$& $[2,2,2,2,2,2]$ &4&$(\mathbb Z/2)^3\times (\mathbb Z/4)^2$\\
$(\mathbb Z/2)^4 $ & $\langle 16,14\rangle$& $[2,2,2,2,2]$&$[2,2,2,2,2]$  &4&$(\mathbb Z/4)^4$\\
$(\mathbb Z/3)^2$ & $\langle 9,2\rangle$& $[3,3,3,3]$& $[3,3,3,3]$&2&$(\mathbb Z/3)^5$\\
$(\mathbb Z/2)^3$ & $\langle 8,5\rangle$& $[2,2,2,2,2]$& $[2,2,2,2,2,2]$  &5&$(\mathbb Z/2)^4\times (\mathbb Z/4)^2$
\end{tabular}	
\caption{ }\label{tab1}
\end{table}
\end{theorem}

\begin{remark}\label{monodromies0}
We read the Hurwitz generating vectors associated to the surfaces in Table \ref{tab1} from   \cite[Section 3]{BaCa04} and \cite[Section 6]{BCG08}.

\begin{enumerate}[label=\roman*),leftmargin=\parindent]
 \setlength\itemsep{8pt}

\item $G=\frak A_5$:

\begin{itemize}[leftmargin=*]
\item $V_1=[(24)(35), (13452),(12345)]$, $V_2=[(123),(345),(243),(152)]$;
\item $V_1=[(12534),(12453),(12345)]$, 
$V_2=[(12)(34),(24)(35),(14)(35),(243)]$;
\item $V_1=[(123),(345),(15432)]$, $V_2=[(12)(34),(13)(24),(14)(23),(14)(25),(14)(25)]$.
\end{itemize}

\item $G=\frak S_ 4 \times \mathbb Z/2$:

 $V_1=[((12),0), ((1234), 1), ((243),1) ]$,
 
$V_2=[((12)(34),1),((12),1),((34),1),((14)(23),1),((23),1),((14),1)]$.

\item $G=G(32)=\langle x_1,x_2,x_3,x_4,x_5\mid
x_j^2,
[x_1,x_2]x_4,[x_1,x_3]x_5,[x_j,x_k]
 \text{ for } (j,k)\neq(1,2),(1,3)
 \rangle $:
 
 $V_1=[ x_2  x_3  x_4,   x_2,   x_1  x_2  x_3 x_5, x_1 x_2 ]$,

 $V_2=  [  x_1    x_4  x_5, x_2  x_3 x_4 x_5,  x_2  x_4  x_5, x_1  x_3  x_4] $.

\item $G=(\mathbb Z/5)^2$: $V_1=[e_1,e_2, -(e_1+e_2) ]$, $V_2=[ e_1-e_2,e_1+2e_2,-2e_1-e_2]$.

\item $G=\mathfrak {S}_4$:
$V_1=[(123),(1234),(1243) ]$, $V_2=[(12),(12),(23),(23),(34),(34)]$.

\item $G=G(16)=(\ZZ/4\times \ZZ/2)\rtimes \ZZ/2=
\langle x,y,z\mid x^4,y^2,z^2, [x,y],[y,z], [z,x]y
\rangle$:

 $V_1=[z,z,x,x^{-1}]$, $V_2=[zx^2y,zx^2y,xyz,(xyz)^{-1}]$.
 
\item $G=D_4\times \mathbb Z/2= \langle r,s \mid r^4,s^2, (sr)^2\rangle \times \ZZ/2$:

$V_1=[(1,1),(s,1),(rs,0),(r,0) ]$, $V_2=[(s,0),(sr,1),(sr^2,0),(sr,1),(r^2,1),(r^2,1) ]$.

\item $G=(\mathbb Z/2)^4 $: $V_1=[e_1, e_2, e_3, e_4, e:=e_1+e_2+e_3+e_4 ]$, $V_2= [ e + e_1, e + e_2, e_1 + e_3, e_2 + e_4 , e_3 + e_4]$.

\item $G=(\mathbb Z/3)^2$: $V_1=[e_1, e_2, -e_1, - e_2]$, $V_2=[
e_1+e_2 , e_1 - e_2, -e_1 - e_2 , -e_1 + e_2]$.

\item $G=(\mathbb Z/2)^3$: $V_1=[e_1+ e_2, e_1 + e_3, e_2+  e_3, e:= e_1 +e_2 + e_3, e]$, $V_2= [e_1, e_2, e_3, e_1, e_2, e_3]$.

\end{enumerate}

\end{remark}

\begin{theorem}\label{record}
Let $S$ be a surface isogenous to a product of unmixed type, with $q(S)= p_g(S)=0$.

Then  $|\Aut_\QQ(S)| \leq  192$,
and equality is attained for $G = (\ZZ/2)^3$.

If $G$ is abelian,  $|\Aut_\QQ(S)| $ reaches    $160$ for $G = (\ZZ/2)^4$,  
$72$ for $G = (\ZZ/3)^2$,  while  $|\Aut_\QQ(S)| =  75$ for the classical Beauville surface (the sixth case in the table).

Moreover,  if  $G$ is not abelian, then $|\Aut_\QQ(S)| \leq  32$. 
\end{theorem}

\begin{proof}
We have already observed that $\Aut_\QQ(S)$ is the subgroup $\Aut^*(S)$ of $\Aut(S)$ not exchanging the two fibrations
$ S \ra C'_j= C_j /G$.

We have also shown that we have an exact sequence 
$$ 1 \ra( \sC_1 \times\sC_2)/ Z(G) \ra \Aut^*(S)  \ra H/\Inn(G) \ra 1.$$

If, for each $C_j \ra C'_j$ there are three points with  local monodromies different from all the other local monodromies,
then our previous argument shows that $\sC_j = Z(G)$, hence  $$( \sC_1 \times\sC_2)/ Z(G) \cong Z(G).$$

\

{\bf The case of the classical Beauville surface.}
Here $G = (\ZZ/5)^2$, and the local monodromies are\footnote{these formulae are valid for 
$G = (\ZZ/n)^2$, for all $n$ relatively prime to 6.}
 the two triples $[e_1, e_2 , -e] : = [(1,0), (0,1), (-1,-1)]$,
respectively 
{
$$\Psi [e_1, e_2 , -e] : = [(1,-1), (1,2), (-2,-1)] = : [e'_1, e'_2, -e'].$$
}
In both cases we have three elements which generate  $G$ with sum equal to zero.

Hence $N'_j \cong \frak S_3$, and the map $N'_j \ra H_j < \Aut(G)$ is obtained if we consider  the induced linear map 
$\Phi$ associated to a permutation of the
three vectors.

{ The five nontrivial maps in $H_1$ are the maps sending $ (a,b)\in G$ respectively to $$ (b,a), \  (-a,-a+b), \ (a-b, -b) , \ (-b, a-b) , \ (b-a, -a).$$

The last two maps correspond to the 3-cycles of $\frak S_3$ and cyclically permute  the vectors  $e'_1, e'_2, -e'$.
On the other hand, looking at the image of $e_1'=(1,-1)$ under  the first  three maps, we obtain the vectors $(-1,1), (-1,-2),(2,1)$, none of which is in the set  $\{e'_1, e'_2, -e'\}$.

Therefore, $H$ has order 3  and $|\Aut^*(S)|= 75$.
}

\

{\bf The case with $G= (\ZZ/2)^4$.}
The  local monodromies are
$ [e_1, e_2, e_3, e_4, e:=e_1+e_2+e_3+e_4]$ and 
$$ \Psi [e_1, e_2, e_3, e_4, e] = : [e'_1, e'_2, e'_3, e'_4, e'] : = [ e + e_1, e + e_2, e_1 + e_3, e_2 + e_4 , e_3 + e_4].$$

We see that the permutation $\s_1$ exchanging $e_1$ with $e_2$ and $e_3$ with $e_4$ yields the
permutation exchanging $e'_1$ with $e'_2$ and $e'_3$ with $e'_4$.

In fact, the first quintuple consists of 4 vectors of weight $1$ and one of weight  $4$, the second 
quintuple consists of 3 vectors of weight $2$ and 2  of weight  $3$. Hence, every non-trivial permutation of $\frak S_4$
permuting the elements $e_1, e_2, e_3, e_4$ leaves the weight invariant, therefore, if it leaves invariant the set  
$ \sE' : = \{ e + e_1, e + e_2, e_1 + e_3, e_2 + e_4 , e_3 + e_4\}$ it must swap $e'_1$ with $e'_2$, hence $e_1$ with $e_2$,
and then $e_3$ with $e_4$.

If instead we have a permutation in  $\frak S_5 \setminus \frak S_4$, then $e$ is exchanged with some $e_j$.

If we exchange $ e$ with $e_1$, then  we must exchange  $e_2 $ with $e_3$, this leaves invariant the set $\sE'$.

If we exchange  $ e$ with $e_3$, then we must permute $\{e_1, e_2\}$ with $\{e_1, e_4\}$,
and the only possibility of having  $\sE'$ invariant is that $e_2$ is permuted with $e_4$, $e_1$ is fixed.

Now, using the two above permutations of order two, and their products with $\s_1$, we see that the orbit
of $e$ equals the whole set $\sE : = \{e_1, e_2, e_3, e_4, e\}$.

Hence the group of permutations of the set $\sE : = \{e_1, e_2, e_3, e_4, e\}$ inducing a linear map
permuting the set $\sE'$  has order 10, hence $H \cong D_{5}$,  and 
 $\Aut^*(S)$ has order equal to  $16 \cdot 10 = 160$.

 \
 
{\bf The case with $G= (\ZZ/3)^2$.}

Here, the four-tuples of monodromies are $ [e_1, e_2, -e_1, - e_2]$, $ [e_1+e_2 , e_1 - e_2, -e_1 - e_2 , -e_1 + e_2]$.

Now, the admissible permutations of the first 4-tuple which induce an automorphism of $G$ are 
leaving the partition  $\{e_1, - e_1\} \cup \{e_2, - e_2\}$ invariant. 

These permutations permute clearly the complementary set $\sE^c$ of $\sE=\{e_1, e_2, -e_1, - e_2\}$
(inside $G \setminus \{0\}$), hence here  $\Aut^*(S)$ has order equal to  $9 \cdot 8 = 72$.

\

The next case that we consider is the record winning case.

\

{\bf The case with $G= (\ZZ/2)^3$.}

Here, the 6-tuple of monodromies is  $ [e_1, e_2, e_3, e_1, e_2, e_3]$, 
while the 5-tuple of monodromies is  $ [e_1+ e_2, e_1 + e_3, e_2+  e_3, e:= e_1 +e_2 + e_3, e]$.

We assume without loss of generality that the first 3 branch points are $\infty, 0, 1$.

We claim that, for a proper choice of the branch points,  the group $N'_1$ surjects onto  the group $\frak S_3$ of automorphisms of $G$, permuting the 
three vectors $e_1, e_2, e_3$.

To achieve this, we need to show that, for each permutation of $\{1,2,3\}$, there is a projectivity permuting accordingly
the pairs $\{\infty, P_4\}$, $\{0 , P_5\}$, $\{1, P_6\}$. 

A way to show this is to observe that  $\sB_1: = \{\infty, 0, 1\}$ is an orbit for the natural subgroup, isomorphic to $\frak S_3$,
preserving $\sB_1$. Then $\sB_2 : = \{P_4, P_5, P_6\}$ must be another orbit, in particular, given the tranposition
$ z \mapsto 1/z$ exchanging $\infty$ and $0$, we may take $P_6$ as the other fixpoint $-1$. Arguing like this, we find
that $P_4, P_5, P_6$ must be  $1/2, 2, -1$.

The kernel of this surjection consists of projectivities $\tau$ fixing each of the sets $\{\infty, P_4\}$, $\{0 , P_5\}$, $\{1, P_6\}$.

Clearly since a projectivity with four fixpoints is the identity,  the kernel is a subgroup of $(\ZZ/2)^3$ such that $\tau$ different from the identity
has  at most one coordinate
different from $1$. Hence the subgroup is either  $(\ZZ/2)(1,1,1)$ or is contained in $(1,1,1)^{\perp}$.

In both cases, if there is such a nontrivial $\tau$, there is a transformation permuting the two elements of a couple,
without loss of generality $1, -1$. The only projectivities $\tau$ with such a property are of the form $ z \mapsto \frac{a z + b }{-b z -a}$, which have order 2. If $a=0$, then we obtain that $0, \infty$ are exchanged, which contradicts that $\tau$ is in the kernel.
Hence we may assume that $a=1$, and then $ \tau (0) =0 \Leftrightarrow b=0$ and so $ \tau(2) = -2$, hence this case is excluded and we must have that $\tau$ transposes $0$ with $2$, that is, $a=1, b= -2$; we see now that $\tau (z) =  \frac{ z -2 }{2 z -1}$
 transposes $0$ with $2$, $\infty$ with $1/2$, and  $1$ with $-1$.
 
 In this way we have found a group $N'_1$ of cardinality $6\cdot 2=12$. 
 
 For each other choice of branch points $P_1, \dots, P_6$ we have a homorphism  $N'_1 \ra \frak S_3$ 
 whose kernel has either 4 or 2 elements. 
 
 The cardinality 12 is maximal unless 
 $N'_1$  surjects onto $\frak S_3$ and the kernel has order $4$.
 
 In this case $N'_1$  would be a non abelian subgroup of $\PP GL(2, \CC)$ of order $24$, hence
 it would be isomorphic to $\frak S_4$. There is indeed a surjection of $\frak S_4$ onto $\frak S_3$
 with kernel the Klein group $\sK$ of double transpositions.
 
 There is according to Klein only one action of  $\frak S_4$ on $\PP^1$, corresponding to a 
 Galois covering of $\PP^1$ branched in three points, with local monodromies of respective orders $2,3,4$.
 
 Then all the orbits have $24$ elements, except the special orbits, with respectively $12,8, 6$ elements.
 
 Hence $\sB$, which is $N'_1$-invariant, is a union of orbits and  should be the orbit with 6 elements, 
 each having as  stabilizer a cyclical permutation of order $4$.
 
The partition of the orbit as the union of three pairs  corresponds to the choice to the three Sylow subgroups of
order $8$, which are dihedral groups $D_4$, inverse images of the three transpositions of $\frak S_3$.
To each pair is then associated an element of the set $\{e_1, e_2, e_3\}$ and in this way the monodromy
homomorphism to $G$ is determined.

Concerning the second monodromy, we see that we have a surjection of $N'_2$ onto $\frak S_3$ which permutes the
three different monodromies,
 and leaves $e$ fixed. And the kernel must be trivial.

In both  cases the previous criterion does not apply, because there are repeated monodromies, but we see that $N'_2$ consists 
of 
$6$ elements, and for each one of them there is a permutation in $N'_1$, consisting of $ 12$ or $24$ elements,  corresponding to the given action on $G$; hence 
$\Aut^*(S)$ has order equal to  $ 8 \cdot 12 $ or $ 8 \cdot 24$, that is, $96$ or $192$ elements.

\

{\bf The non-abelian cases.}

 We shall give more crude estimates for the case where the group $G$ is not abelian.

Recalling that $H = H_1 \cap H_2 \supset \Inn(G)$, 
we have $\frac{ H}{\Inn (G)}  \subseteq \frac{  H_j }{\Inn(G)}$
and   the exact sequence
\[1\ra \sC_j/Z(G) \ra N_j'\ra H_j/\Inn(G)\to 1,\]
which together  with \eqref{eq_autS}  provide the  rough estimation:

\[|\Aut^*(S)|\leq   \min\{|N_1'|\cdot |\sC_2| , |N_2'|\cdot |\sC_1| \}.\]

The group $N'_j$, being a finite subgroup of $\PP GL(2, \CC)$ is either:
\begin{enumerate}
\item
cyclic $\ZZ/n$, with two orbits of cardinality 1, the others of cardinality $n$;
\item
dihedral $D_n$, with three special orbits of cardinalities $(n,n,2)$ and the  others of cardinality $2n$;
\item
$\frak A_4$, with three special orbits of cardinalities $(6,4,4)$ and the  others of cardinality $12$;
\item 
$\frak S_4$, with three special orbits of cardinalities $(12, 8,6)$ and the  others of cardinality $24$;
\item
$\frak A_5$, with three special orbits of cardinalities $(30,20,12)$ and the  others of cardinality $60$.
\end{enumerate}

Since $\sB_j$ is a union of orbits for $N_j'$, we observe that 
if there is an orbit of length 1, then $N'_j$ must be cyclic.

Looking at Table \ref{tab1}, we see that the case  $\frak A_5$ never happens, and if $N_j'=\frak S_4  $, then $\sB_j$ is a single orbit with $6$ elements, and
the six monodromies have the same order.

Similarly, if $N_j'=\frak A_4$, 
then $\sB_j$ is a single orbit with either $4$ or $6$ elements, and in both cases all branching order are equal.

In the case where $N'_j$ is cyclic, then its order is bounded by  $|\sB_j|$ if all the branching orders are equal,
otherwise by $|\sB_j|-1$.

Similarly, if $N'_j$  is dihedral, its order is bounded by  $2| \sB_j|$ if all the branching orders are equal, 
otherwise by $2|\sB_j|-4$.

Using these considerations we obtain the following estimations:
\begin{itemize}
\item if $G=\frak A_5$, $T_1=[2,5,5]$ and $T_2=[3,3,3,3]$, then 
$|N_1'|\leq 2$ and $|N_2'|\leq 12$.

\item If $G=\frak A_5$, $T_1=[5,5,5]$ and $T_2=[2,2,2,3]$, then 
$|N_1'|\leq 6$ and $|N_2'|\leq 3$.

\item If $G=\frak A_5$, $T_1=[3,3,5]$ and $T_2=[2,2,2,2,2]$, then 
$|N_1'|\leq 2$ and $|N_2'|\leq 10$.


\item If $\mathfrak S_ 4 \times \mathbb Z/2$, $T_1=[2,4,6]$ and $T_2=[2,2,2,2,2,2]$, then 
$|N_1'|=1$ and $|N_2'|\leq 24 $.

\item If $G=G(32)$, $T_1=[2,2,4,4]$ and $T_2=[2,2,2,4]$, then 
$|N_1'|\leq 4$ and $|N_2'|\leq 3$.

\item If $G=\frak S_4$, $T_1=[3,4,4]$ and $T_2=[2,2,2,2,2,2]$, then 
$|N_1'|\leq 2$ and $|N_2'|\leq 24$.

\item If $G=G(16)$, $T_1=[2,2,4,4]$ and $T_2=[2,2,4,4]$, then 
$|N_1'|\leq 4$ and $|N_2'|\leq 4$.

\item If $G=D_4\times \mathbb Z/2$, $T_1=[2,2,2,4]$ and $T_2=[2,2,2,2,2,2]$, then 
$|N_1'|\leq 3$ and $|N_2'|\leq 24$.
\end{itemize}

We look now at the cardinalities 
$|\sC_1| , |\sC_2 |$, and  recall that $Z(G)$ is trivial for $\frak A_5$ and $\frak S_4$, has order $2$ for $\frak S_4 \times \ZZ/2$,
and equals $(\ZZ/2)^2$ for $G(32), G(16)$ and $D_4  \times \ZZ/2$.

As remarked above, if
 for $C_j \ra C'_j$ there are three points with  local monodromies different from all the other local monodromies,
then  $\sC_j = Z(G)$.

By inspecting the  local monodromies  reported in  Remark \ref{monodromies0}, we see that 
$|\sC_j|  = |Z(G)|$ ($j=1,2$), except possibly if 
$G=\frak S_4$ or $G=G(16)$ or $G=D_4\times \ZZ/2$.

Thus, if $G$ is not one of these 3 groups, we achieve the bound $|\Aut^*(S)|\leq 12$.

For $G=\frak S_4$, it holds  $\sC_1=Z(G)$, and $|\Aut^*(S)| \leq |N_2'|\leq 24$.

Let us consider $G=G(16)$, whose centre is generated by  $x^2,y$ (see Remark \ref{monodromies0}).
The local monodromies are $[z,z,x,x^{-1}]$ and $[zx^2y,zx^2y,xyz,(xyz)^{-1}]$, so that  $|\sC_j/Z(G)|\leq 2$, whence 
$|\Aut^*(S)|\leq 32$.

Finally, let us consider $G=D_4\times \ZZ/2$, whose centre has order 4 and the local monodromies are:
$$V_1=[(1,1),(s,1),(rs,0),(r,0) ], \ V_2=[(s,0),(sr,1),(sr^2,0),(sr,1),(r^2,1),(r^2,1) ].$$
 
so that  $\sC_1=Z(G)$  and   $|\sC_2/Z(G)|\leq 2$,
 whence $|\Aut^*(S)|\leq 24$.

 \end{proof}

 We consider now $\Aut_\ZZ(S)$.

By Remark \ref{trivial}, it follows that $\Aut_{\ZZ}(S) < Z(G) = N^0(G) \cap (G \times G) / \De_G$, 
because  
the image of  $\Aut_\ZZ(S)$ in $N'_j=N_j/G$ is trivial  for $j=1,2$.
This means, that $\Aut_{\ZZ}(S)$ is a group of automorphisms induced by a subgroup $H < Z(G) \times\{1_G\} < G \times G$.

\begin{theorem}\label{AutZ0}
Let $S$ be a surface isogenous to a product of unmixed type, with $q(S)= p_g(S)=0$.

Then  $\Aut_\ZZ(S)$ is trivial, except for the case of $G=D_4\times \ZZ/2$, where  $\Aut_\ZZ(S) \cong \ZZ/2$
and the intersection form is even.
\end{theorem}

\begin{proof}
The statement is clear for the cases with $G=\frak A_5$ and $G=\frak S_4$.

In the remaining case, we use a MAGMA \cite{MAGMA} script  to determine which  automorphisms in $Z(G)$ act trivially on $H_1(S,\ZZ)$ (see Appendix \ref{script}).

The computations below show that  there are no non-trivial elements  acting trivially on $H_1(S,\ZZ)$ in all cases except for $G=D_4\times \ZZ/2$.

For $G=D_4\times \ZZ/2$ the element $(r^2,0)$ acts trivially on  $H_1(S,\ZZ) \cong (\ZZ/2)^3\times (\ZZ/4)^2$.

There remains to show that  $(r^2,0)$ acts trivially on $H^2(S,\ZZ)$.

In view of the exact sequence
$$ 1 \ra \Tors(S) \cong (\ZZ/2)^3\times (\ZZ/4)^2 \ra H^2(S,\ZZ) \ra \Num(S) = H^2(S,\ZZ)/ \Tors(S) \cong \ZZ^2\ra 0,$$
and since  $(r^2,0)$ acts trivially on $\Tors(S) $, and on the quotient group $\Num(S)$,
it suffices to find a splitting $\Num(S) \ra H^2(S,\ZZ)$ which is  $(r^2,0)$-invariant.

As in \cite{ccck} we consider the two fibrations $f_i : S \ra C_i /G$, and denote by $F_i$ the respective fibres,
and argue as follows.

Because of the multiplicities of the fibres, we find $Z(G)$ invariant divisors $\Phi_1, \Phi_2$ such that 
the following linear equivalences hold true:
$ F_1 \equiv 4 \Phi_1, F_2 \equiv  2 \Phi_2$.

Observe that $\Phi_1 \cdot \Phi_2 = |G|/ 8 = 2.$

Moreover $K_S$ is numerically equivalent to $\Phi_1 + 2 \Phi_2$, hence  the intersection form on $S$ is even
if the numerical class of $\Phi_1$ is divisible by 2, and then if $ 2 \Phi'_1$ is numerically equivalent to
$\Phi_1$, then $\Phi'_1 , \Phi_2$ are a basis of $\Num (S)$, with $\Phi'_1 \cdot \Phi_2 =1, (\Phi'_1)^2 = \Phi_2^2 =0$.

To find $\Phi'_1$, consider the surjection $ G = D_4 \times \ZZ/2 \twoheadrightarrow (\ZZ/2)^3$, and its kernel $H
= \langle r^2 \rangle \cong \ZZ/2$.

Then, setting $$B = C_1 / (\ZZ/2) ,  E : = C_1 / H,$$
 $B$ is a $D_4$-covering of $\PP^1$
with branching multiplicities $(2,2,4)$, hence $B$ is a projective line, while $E$ is an elliptic curve, a 
$(\ZZ/2)^3$ covering of $\PP^1$ branched in 4 points, and $C_1 \ra E$ is branched in 4 points.

Hence $C: = C_1$ is a hyperelliptic curve of genus $3$, and $ K_C  = 2 \sH$, where $\sH$ is the hyperelliptic divisor
(inverse image of a point in $B$), whose linear equivalence class  is invariant for every automorphism of $C$.

We consider then  a divisor  $ \hat{\Phi}_1$ on $C_1 \times C_2$ corresponding to the 
divisor $  \sH$ on $C_1$.
More precisely, consider the $G$-orbit $\Oh$ in $C$ corresponding to the last branch point of $ C \ra \PP^1$,
which is indeed the $G / \langle r\rangle \cong \ZZ/2 \times \ZZ/2$-orbit.

$ \Oh = \{ P, (0,1)P, (s,0) P, (s,1)P\}$ consists of the ramification points for $C \ra E$, hence
$$K_C \equiv (P+ (0,1)P) + ( (s,0) P +  (s,1)P) \equiv \sH + \sH,$$
since $\sH$ is the pull back of a point in $B \cong \PP^1$.

In particular, $$(P+ (0,1)P) \equiv  ( (s,0) P +  (s,1)P).$$

On $C_1 \times C_2$ consider the divisor $ \hat{\Phi}_1: = P \times C_2 + (0,1)P\times C_2.$

For each automorphism in $ G \times G$
the divisor $ \hat{\Phi}_1$ is either left fixed, or sent to $(s,0)  \hat{\Phi}_1$: then 
the previous linear equivalence shows that the linear equivalence class of 
$ \hat{\Phi}_1$ is left invariant.

By our choice, the effective divisor $ \hat{\Phi}_1$ on $C_1 \times C_2$ is invariant for the action of the
group
 $G'': = \langle r \rangle \times \ZZ/2$, hence the corresponding line bundle has a $G''$-linearization,
 therefore $ \hat{\Phi}_1$ descends to a divisor $\Phi''_1$ on $S'' : = (C_1 \times C_2) / G''$
 of which it is the inverse image.

 Likewise $ \hat{\Phi}_1 + (s,0)  \hat{\Phi}_1$ is $G$-invariant, and is the full inverse image of
 $\Phi_1$.
 
 We have an unramified double cover $ p: S'' \ra S = S'' / s$, and both $\Phi''_1, s \Phi''_1$ 
 satisfy  $$p_* (\Phi''_1 ) = p_*  (s \Phi''_1) = \Phi_1.$$

 Now, consider the line bundle $\sL : = \hol_{S''}( \Phi''_1)$. Then, since also $ (s,0)  \hat{\Phi}_1$ is $G''$-invariant
 and linearly equivalent to $\hat{\Phi}_1$, 
 $ s^*(\sL) \cong \sL$,
 and we can consider the associated Theta group 
 $$ 1 \ra \CC^* \ra \Theta(\sL) \ra \ZZ/2 = \langle s \rangle  \ra 1,$$
 which is classified as a central extension by $\e \in H^2(\ZZ/2 , \CC^*) = 1.$
 
 Hence the Theta group splits and there is a a lifting of $s$ to an automorphism of $\sL$,
 which yields a quotient line bundle on $S$, hence a  divisor class $\Phi'_1$ on $S$ such that $p^*(\Phi'_1) \equiv \Phi''_1.$
 
 But then $2 \Phi'_1 = p_*(p^*(\Phi'_1)) \equiv p_*( \Phi''_1) = \Phi_1.$

\end{proof}
\begin{remark}
That the intersection form is even for $D_4 \times \ZZ/2$, this was proven also by Kyoung-Seog-Lee
and his collaborators (personal communication).
\end{remark}


\section{Surfaces with $q \geq 1$ and  nontrivial $\Aut_{\QQ}(S)$}\label{irregular}

\subsection{Irregularity $q=2$}
In the paper \cite{clz} it was shown that, for a surface $S$ with $q(S) \geq 2$, $\Aut_{\QQ}(S)$ is trivial, 
unless $q(S)=2$ and the surface $S$ is isogenous to a product of unmixed type, with $C'_j = C_j /G$
having genus $1$ for $j=1,2$,  the group $G$ is abelian with 
$$G = (\ZZ/2m) \oplus (\ZZ/2mn), \quad {\rm or } \quad \ G = (\ZZ/2) \oplus (\ZZ/2m) \oplus (\ZZ/2mn),$$
and all the branch multiplicities are equal to $2$. 

Using the notation of \cite{clz} we let  the number of branch points be $2k$ for the first covering, and $2l$ for the second covering.
Moreover, the local monodromies of the first covering  are all equal to $\ga$, while the local monodromies of the second covering all are equal to $\ga'$, with $\ga \neq \ga'$. 

If $k \geq 2$, then $N'_1$ is trivial, similarly $N'_2$ is trivial if $l \geq 2$.

In loc. cit. it was proven that $\Aut_{\QQ}(S)$ is the involution $\iota$ induced by $\ga \times \ga' \in G \times G$.

We shall now see whether $\iota$ belongs to $\Aut_{\ZZ}(S)$.

\begin{question}
If $S$ is an irregular surface with $q(S)=2$, then can $\Aut_{\ZZ}(S) $ be nontrivial?
Equivalently, is  it possible that $  \Aut_{\ZZ}(S) = \langle \iota \rangle \cong \ZZ/2  $?
\end{question}

The first author, in joint work with Gromadtzki in summer 2015 (\cite{cg2}, or handwritten notes by the first author), has studied the action of $\iota$
on $H_1(S,\ZZ)$, showing that for certain values of $m,n$ the action is trivial, and for others it is nontrivial
(for the first case, triviality amounts to $m$ being even).

The difficulty is however to determine  the triviality of the action on $H^2(S, \ZZ)$, because the 
above surfaces have large  $p_g$. 

\subsection{Irregularity $q=1$}
The paper \cite{CL18} showed that for a surface $S$ with $q(S) =1$, 
if $\Aut_{\QQ}(S)$ is non-trivial, then it has order at most 4, and that if equality holds, then
 $S$ is isogenous to a product of unmixed type; the authors  also give examples where
$\Aut_{\QQ}(S)$ has order 4.

If $q=1$, then any element $ h \in \Aut_\ZZ(S)$ acts as the identity on the elliptic curve $A = \Alb(S)$, by 
\cite[Lemma 2.1]{CFGLS24}.

  Assume now that $S$ is isogenous to a product of unmixed type:  is the case  
$|\Aut_\ZZ(S)|=4$, $q(S)=1$  possible?

The surface $S$ admits a second fibration $ f : S \ra C_2/G = \PP^1$, which has 
 multiple fibres of multiplicities $m_1 \leq m_2\leq  \dots \leq m_r$ with $r \geq 3$, and $h$ acts as the identity on
$C_2/G$, by  \cite[Principle 3]{CatLiu21} unless $m_1=m_2=2$, $m_j$ is odd for $j=3,4$,  $r = 4$ and
the action on $C_2/G$ has order 2.

Hence, if $\Aut_{\ZZ}(S)$ had order 4, there would be an element $h=(h_1,h_2)$ acting as the identity on
$C_2/G$,  and acting as the identity on $C_1$, so that  necessarily  $h_2 \in Z(G)$.

But, again, how to deal with the transcendental cycles in $H^2(S, \ZZ)$?

\bigskip

 

\appendix

\section{MAGMA scripts}\label{script}

The script can be run at \url{http://magma.maths.usyd.edu.au/calc/}.\\

\begin{code_magma}
/* Input: i) the group: G;  ii) the monodromy images of pi_orb(C_i)-> G: mon1 and mon 2;  
iii) the genus of the quotient curves C_i/G: h1 and h2;

Output: the elements of Z(G), acting trivially on H_1(S,Z)*/

// Given the monodromy images, the next function constructs 
// the group pi^orb and the monodromy map pi^orb-->>G

Orbi:=function(gr,mon, h)
F:=FreeGroup(#mon);  Rel:={}; G:=Id(F);
for i in {1..h} do G:=G*(F.(2*i-1)^-1,F.(2*i)^-1); end for;
for i in {2*h+1..#mon} do G:=G*F.(i); Include(~Rel,F.(i)^(Order(mon[i])));  end for;
Include(~Rel,G); 
P:=quo<F|Rel>;
return P, hom<P->gr|mon>;
end function;

// MapProd computes given two maps f,g:A->B the map product induced by the product on B

MapProd:=function(map1,map2)
seq:=[];
A:=Domain(map1); 
B:=Codomain(map1);
if Category(A) eq GrpPC then 
	n:=NPCgens(A);
	else n:=NumberOfGenerators(A); 
end if;
for i in [1..n] do Append(~seq, map1(A.i)*map2(A.i)); end for;
return hom<A->B|seq>;
end function;

// This is the main routine of the script

TrivialActionH1:=function(G,mon1, h1, mon2, h2)
// First of all we construct the monodromy maps  pi_j^orb-->>G

T1,f1:=Orbi(G,mon1, h1);
T2,f2:=Orbi(G,mon2, h2);

T1xT2,inT,proT:=DirectProduct([T1,T2]);
GxG,inG:=DirectProduct(G,G);
Diag:=MapProd(inG[1],inG[2])(G);
f:=MapProd(proT[1]*f1*inG[1],proT[2]*f2*inG[2]);
Pi1S:=Rewrite(T1xT2,Diag@@f); // This is the fundamental group of S=(C1xC2)/G
H1S,q:=AbelianQuotient(Pi1S); //H_1(S,Z)

triv:={G!1};

for z in Center(G) do    // we consider the automorphisms in Z(G) and their action on H_1(S,Z)
	act:={}; // in act we collect the differences z*h-h, where h is a generator
	 of H_1(S,Z)
	z1:=(inG[1](z)@@f);  // lifts of z=(z,1) to the productT1xT2
	for h in Generators(H1S) do
		h1:=h@@q;     Include(~act,q(z1*h1*z1^-1)-h);  // con
	end for;  
	// check if g in Z(G)=Aut_Q(S), acts trivially on  H_1(S,Z)
	if act eq {H1S!0} then Include(~triv,z); end if;	
end for; 

return triv; 
end function;

\end{code_magma}

\subsection{The case $p_g=q=0$}

We use the previous script, to check, which  elements of $Z(G)$ act non-trivially on $H_1(S,\ZZ)$, for the algebraic data listed in Remark \ref{monodromies0}.

\begin{code_magma}
G:=SmallGroup(48,48); // S4 x Z/2
mon1:=[ G.1 * G.4 * G.5, G.1 * G.2 * G.3^2 * G.5, G.2 * G.3 ];
mon2:=[ G.2 * G.5, G.2 * G.4 * G.5, G.1 * G.2 * G.3^2 * G.4, G.1 * G.2 * G.3 * G.5, 
G.1 * G.2 * G.3 * G.5, G.1 * G.2 * G.3^2 ];
TrivialActionH1(G,mon1,0,mon2,0);
> { Id(G) }

G:=SmallGroup(32,27); // G(32)
mon1:=[ G.2 * G.3 * G.4, G.2, G.1 * G.2 * G.3  * G.5, G.1 * G.2 ];
mon2:=[ G.1 * G.4 * G.5, G.2 * G.3 * G.4 * G.5, G.2 * G.4 * G.5, G.1 * G.3 * G.4];
TrivialActionH1(G,mon1,0,mon2,0);
> { Id(G) }

G:=SmallGroup(25,2);   // (Z/5)^2
mon1:=[ G.1, G.2, G.1^4 * G.2^4 ];
mon2:=[ G.1*G.2^4, G.1 * G.2^2, G.1^3 * G.2^4 ];
TrivialActionH1(G,mon1,0,mon2,0);
> { Id(G) }

G:=SmallGroup(16,3); // G(16)
mon1:=[ G.2, G.2, G.1, G.1^3 ];
mon2:=[ G.2 * G.1^2* G.3 , G.2 * G.1^2* G.3, G.1*G.3*G.2, (G.1*G.3*G.2)^-1 ];
TrivialActionH1(G,mon1,0,mon2,0);
> { Id(G) }

G:=SmallGroup(16,11); // D4 x Z/2
mon1:=[ G.2 * G.3, G.3 * G.4, G.1 * G.3, G.1 * G.2 * G.3 * G.4 ];
mon2:=[ G.1 * G.4, G.1 * G.4, G.3, G.3, G.2 * G.4, G.2 * G.4 ];
TrivialActionH1(G,mon1,0,mon2,0);
> { Id(G), G.4 } // G.4=(r^2,0) acts trivially

G:=SmallGroup(16,14); // (Z/2)^4
mon1:=[ G.1, G.2, G.3, G.4, G.1 * G.2 * G.3 * G.4 ];
mon2:=[ G.2 * G.3 * G.4, G.1 * G.3 * G.4, G.1 * G.3, G.2 * G.4, G.3 * G.4 ];
TrivialActionH1(G,mon1,0,mon2,0);
> { Id(G) }

G:=SmallGroup(9,2); // (Z/3)^2
mon1:=[ G.1, G.2, G.1^2, G.2^2 ];
mon2:=[ G.1 * G.2, G.1 * G.2^2, G.1^2 * G.2^2, G.1^2 * G.2 ];
TrivialActionH1(G,mon1,0,mon2,0);
> { Id(G) }

G:=SmallGroup(8,5); // (Z/2)^3
mon1:=[ G.1 * G.2, G.1 * G.3, G.2 * G.3, G.1 * G.2 * G.3, G.1 * G.2 * G.3  ]; 
mon2:=[ G.1, G.2, G.3, G.1, G.2, G.3 ];
TrivialActionH1(G,mon1,0,mon2,0);
> { Id(G) }

\end{code_magma}


\end{document}